\documentclass{article} 

\usepackage{slashbox}

\usepackage{latexsym,amssymb,a4wide,amscd}
\usepackage{amsmath}
\usepackage[latin1]{inputenc}
\usepackage[T1]{fontenc}

\usepackage{array}

\usepackage{enumerate}

\usepackage{amsthm}

\usepackage{lscape}

\usepackage{url}

\usepackage{mathabx}

\usepackage[mathcal]{euscript}

\usepackage{longtable}
\usepackage{multirow}

\usepackage{MnSymbol}

\usepackage{amsthm}
\numberwithin{equation}{section}
\numberwithin{figure}{section}
\theoremstyle{plain}
\newtheorem {theorem}{Theorem}[section]

\newtheorem {lemma}[theorem]{Lemma}
\newtheorem {prop}[theorem]{Proposition}
\newtheorem {kor}[theorem]{Corollary}

\theoremstyle{definition}
\newtheorem {definition}[theorem]{Definition}
\theoremstyle{remark}
\newtheorem {remark}[theorem]{Remark}

\usepackage[pdftex]{graphicx} 

\usepackage{pdfpages}

\usepackage[nooneline]{caption}

\begin{document}

\date{}

\title{\Large {\bf A necessary condition for the tightness of odd-dimensional combinatorial manifolds}}

\author{Jonathan Spreer\thanks{School of Mathematics and Physics, The University of Queensland, Brisbane QLD
4072, Australia. j.spreer@uq.edu.au}}

\maketitle

\subsection*{\centering Abstract}

{\em
We present a necessary condition for $(\ell-1)$-connected combinatorial 
$(2\ell +1)$-manifolds to be tight.
As a corollary, we show that there is no tight combinatorial three-manifold
with Betti number at most two other than the boundary of the four-simplex and
the nine-vertex triangulation of the three-dimensional Klein bottle.
}

\medskip
\noindent
\textbf{MSC 2010: } 
{\bf 57Q15};  
57N10; 
05A19 

\medskip
\noindent
\textbf{Keywords: } combinatorial manifolds, rsl-functions, slicings, tightness, tight triangulations, $\sigma$-vector

\section{Introduction}
\label{sec:prelim}

Tight combinatorial manifolds are rare but very special objects.
There are strong necessary conditions on when a combinatorial manifold
can be tight and it is conjectured that all tight combinatorial manifolds
are strongly minimal triangulations \cite[Conjecture 1.3]{Kuehnel99CensusTight}.

On the other hand, given a combinatorial manifold $M$ it is difficult to
check in general whether or not $M$ is tight. One way to do this would
be to look at all regular simplex-wise linear functions on $M$ and
check if they all have the minimum number of critical points, i.e.,
if they are all {\em perfect}, see \cite{Bagchi14StellSpheresTightness} for an elaborate way
to do this.
As a consequence, necessary as well as sufficient conditions for tightness are
highly sought after.

\medskip
Here we establish new necessary conditions for the tightness
of odd-dimensional combinatorial manifolds by analysing topological
properties of slicings, i.e., co-dimension one normal sub-manifolds,
which do not depend on the topology of the surrounding manifold. 

As a result, we present
upper bounds on the number of vertices of a combinatorial manifold $M$ 
in terms of its Betti numbers, this way disqualifying large classes of 
topological manifolds from having tight triangulations at all.

In particular we prove the following result about $(\ell-1)$-connected
combinatorial $(2\ell +1)$-manifolds complementing the
results about $(\ell-1)$-connected combinatorial $2\ell$-manifolds
due to K\"uhnel \cite{Kuehnel95TightPolySubm}. 

\begin{theorem}
	\label{thm:dimOdd}
	Let $\mathbb{M}$ be an $\mathbb{F}$-orientable compact closed 
	$(\ell-1)$-connected $(2\ell+1)$-manifold
	admitting an $n$-vertex triangulation which is tight with respect to the field $\mathbb{F}$. Then
	\begin{equation}
		\label{eq:main}
		\beta_{\ell} (M,\mathbb{F}) = \beta_{\ell+1} (M,\mathbb{F}) \geq \left \lceil (-1)^{\ell+1} \frac{(1-\lfloor n/2 \rfloor)_{\ell+1} (1-\lceil n/2 \rceil)_{\ell+1}}{(\ell+1)! \,\, (1-n)_{\ell+1}} \right \rceil
	\end{equation}
	where $(a)_n = a \cdot (a+1) \cdot (a+2) \cdot \ldots \cdot (a+n-1)$ denotes the Pochhammer symbol. 
\end{theorem}

As of today, the known cases of equality in (\ref{eq:main}) are the boundary of the simplex ($\ell \geq 1$, $\beta_{\ell} = 0$) and
the $13$-vertex triangulation of $SU(3) / SO(3)$ ($\ell = 2$ and $\beta_{\ell} = 1$).

\medskip
As a direct consequence any ($\mathbb{F}$-)tight connected combinatorial three-manifold $M$ with $\beta_1 (M,\mathbb{F}) \leq 2$ 
cannot have more than $12$ vertices. Together with further results presented in Section~\ref{sec:sigma} 
and extended computer experiments this leads to the following.

\begin{kor}
	\label{kor:main}
	The boundary of the simplex and the nine-vertex three-dimensional Klein Bottle $S^2 \dtimes S^1$ are the only
	tight combinatorial three-manifolds with first Betti number at most two.
\end{kor}

\subsection*{Acknowledgement}

This work was supported by DIICCSRTE, Australia, under the Australia-India Strategic 
Research Fund (project AISRF06660). Furthermore, the author wants to thank Ole Warnaar 
for introducing him to hypergeometric sums and for the proof of Lemma~\ref{lem:sums}.

\section{Preliminaries}
\label{sec:prelim}

\subsection{Combinatorial manifolds}

A combinatorial $d$-manifold $M$ is an abstract pure simplicial complex of dimension
$d$ such that all vertex links are triangulated standard $PL$-spheres. The {\it $f$-vector} 
of $M$ is a $(d+1)$-tuple $f(M)=(f_0, f_1, \ldots f_d)$
where $f_i$ denotes the number of $i$-dimensional faces of $M$. The 
zero-dimensional faces of $M$ are called {\em vertices}, the one-dimensional faces
are called {\em edges} and the $d$-dimensional faces are referred to as {\em facets}. 
The set of vertices of $M$ will be denoted by $V(M)$ or just $V$ if $M$ is given by the context.

We call $M$ {\it $k$-neighbourly}, if $f_{k-1} = { f_0 \choose k }$, i.e., 
if it contains all possible $(k-1)$-dimensional faces. 
An $n$-vertex combinatorial $d$-manifold $M$ distinct from the boundary of the $(d+1)$-simplex
can be at most $(\lfloor \frac{d+2}{2} \rfloor)$-neighbourly. In this case 
the $f$-vector of an odd-dimensional combinatorial manifold $M$ 
is already determined to be the one of the boundary complex
of the (even-dimensional) cyclic $(d+1)$-polytope with $n$ vertices. This statement
is known as the {\em Upper Bound Theorem} due to Novik \cite{Novik98UBTHomMnf} and Novik and Swartz
\cite{Novik08SocBuchsMod}.

Given a combinatorial manifold $M$ with vertex set $V(M)$ and $W \subset V(M)$, the simplicial complex
$$ M[W] = \{ \sigma \in M \,|\, V(\sigma) \subset W \}, $$
i.e., the simplicial complex of all faces of $M$ with vertex set in $W$, 
is called the {\em sub-complex of $M$ induced by $W$}.
 
\subsection{Tightness}

{\em Tightness} is a condition on subsets of Euclidean space generalising
the notion of convexity: an object is tight if it is ``as convex as possible'', i.\ e.,
as simple as possible, given its topological constraints. More precisely we have
the following.

\begin{definition}[Tightness \cite{Kuehnel95TightPolySubm}]
	\label{def:tightness}
	A compact connected subset $M \subset E^d$ is called 
	{\em $k$-tight} with respect to a field $\mathbb{F}$ if for every open 
	or closed half space $h \subset E^d$ the induced homomorphism
	\begin{equation*} 
		H_{k} (h \cap M, \mathbb{F}) \to H_{k}(M,\mathbb{F}) 
	\end{equation*}
	is injective. If $M \subset E^d$ is $k$-tight with respect to $\mathbb{F}$ 
	for all $k$, $0 \leq k \leq d$,	it is called {\em tight}.
\end{definition}

Here and in the following $H_{\star}$ denotes an appropriate homology theory (i.e., simplicial
homology for our purposes).
An $n$-vertex combinatorial manifold $M$ with vertex set $V = \{ v_1 , \ldots , v_n \}$ 
is said to be {\em tight} if and only
if its canonical embedding
	$$ i: M \to \mathbb{R}^n ; \qquad \qquad v_i \mapsto e_i $$
is tight with respect to at least one field $\mathbb{F}$. 

Alternatively, there is the following
combinatorial definition of tightness for combinatorial manifolds or even arbitrary abstract simplicial complexes.

\begin{definition}[Tightness \cite{Kuehnel95TightPolySubm,Bagchi14StellSpheresTightness}]
	Let $M$ be a combinatorial manifold with vertex set $V(M)$ and let $\mathbb{F}$ be a field. We say
	that $M$ is tight with respect to $\mathbb{F}$ if {\em (i)} $M$ is connected, 
	and {\em (ii)} for all subsets $W \subset V(M)$ $N$ and for all $0 \leq k \leq d$
	the induced homomorphism 
	$$ H_k (M[W], \mathbb{F}) \to H_k (M, \mathbb{F}) $$
	is injective.
\end{definition}

\subsection{Rsl-functions and slicings}
\label{ssec:rsl}

Various discretisations of the concept of Morse theory provide 
important and powerful tools to investigate combinatorial manifolds. 
One of these is based on the following construction.

\begin{definition}[Rsl-function \cite{Kuehnel95TightPolySubm}]
	\label{def:rsl}
	Let $M$ be a combinatorial $d$-manifold. A function $g:M \to \mathbb{R}$ 
	is called \textit{regular simplexwise linear (rsl)} if $g(v) \neq g(w)$ 
	for any two vertices $w \neq v$ of $M$, and $g$ is linear when 
	restricted to a simplex of $M$.
\end{definition}

A point $x \in M$ is said to be {\it critical of index $i$} for an rsl-function $g:M \to \mathbb{R}$ 
if \[H_{i} (M_x , M_x \backslash \{ x \} , \mathbb{F}) \neq 0 \] where 
$M_x := \{ y \in M \, | \, g(y) \leq g(x) \}$ and $\mathbb{F}$ is a field. 
Furthermore, a critical point of index $i$ is said to have {\em multiplicity $m$} if
$m = \operatorname{rk} H_i (M_x , M_x \backslash \{ x \} , \mathbb{F})$.
It follows that no point of $M$ can be critical except possibly the vertices.

Using the notion of rsl-functions and critical vertices as defined above
discrete analogues of the principal results of classical Morse 
can be obtained for combinatorial manifolds 
\cite{Effenberger09StackPolyTightTrigMnf,Kuehnel95TightPolySubm}.
In particular, for any field $\mathbb{F}$ the sum of all critical points of an 
rsl-function $g: M \to \mathbb{R}$ counted
with multiplicity is greater or equal to the sum of all Betti numbers $\beta_i (M,\mathbb{F})$,
$0 \leq i \leq d$, and $g$ is called {\em perfect} in the case of equality.
Moreover, we have the following:

\begin{definition}[Slicing]
	\label{def:slicing}
	\noindent
	Let $M$ be a combinatorial $d$-manifold, 
	$g:M \to \mathbb{R}$ an rsl-function, and $x \in \mathbb{R}$ such that $x \not \in f(v)$ for
	any vertex $v \in V(M)$. Then the pre-image 
	$S(g,x) = g^{-1} (x) \subset M$ is referred to as a {\it slicing} of $M$.
\end{definition}

By construction, a slicing is a (polyhedral decomposition of a) $(d-1)$-manifold and for any ordered pair 
$x < y$ the slicing $g^{-1} (x)$ is isomorphic to $g^{-1} (y)$ whenever 
$g^{-1}([x,y])$ doesn't contain a vertex of $M$. In the following we will denote 
the class of isomorphic slicings between two adjacent vertices of $g$ 
by $S(g,v)$ where $v$ is the vertex with the largest $g$-value smaller than $x$.
It follows that any $n$-vertex
combinatorial manifold has exactly $2^{n-1}-1$ such classes of slicings (not that
the empty set is not counted as a slicing).

Let $S(g,v) \subset M$ be a slicing of a combinatorial manifold $M$, then
$M \setminus S(g,v)$ splits into two connected components $M^-$ and $M^+$
where $M^-$ denotes all points of $x \in M$ with $g(x) \leq g(v)$ and
$M^+$ denotes all $y \in M$ such that $g(y) > g(v)$.

We can now use the theory of slicings and rsl-functions to give yet another definition
of a tight combinatorial manifold.

\begin{definition}[Proposition 3.17 \cite{Kuehnel95TightPolySubm}]
	\label{def:tight3}
	A combinatorial manifold $M$ is tight with respect to some field $\mathbb{F}$
	if and only if all rsl-functions $M \to \mathbb{R}$ are perfect with respect to
	$\mathbb{F}$.
\end{definition}

\subsection{Hypergeometric sums}

Hypergeometric sums are a standard tool to prove combinatorial identities.
Here, we will focus on their definition and some basic properties necessary
to prove Theorem~\ref{thm:dimOdd}. For a more thorough introduction into the subject
see \cite{Andrews99SpecialFunctions,Bailey64GenHypSer}.

Given the {\it Pochhammer symbol} $(a)_n = a (a + 1) \ldots (a+n-1)$ the generalised
hypergeometric sum is defined as
$$ 
	_{r}F_{s} \left ( \begin{array}{c} a_1 , \ldots , a_r \\ b_1 , \ldots , b_s \end{array}; z \right ) = 
	\sum \limits_{n \geq 0} \frac{(a_1)_n \dots (a_r)_n}{(b_1)_n \dots (b_s)_n} \frac{z^n}{n!} ,
$$
and we have
\begin{equation}
\label{eq:1}
{a \choose n} = (-1)^n \frac{(-a)_n}{n!},
\end{equation}
and
\begin{equation}
\label{eq:2}
(a)_{2n} = 4^n (a/2)_n ((a+1)/2)_n .
\end{equation}
Furthermore, for any positive integer $n$ we have
\begin{equation}
\label{eq:chu}
_{2}F_{1} \left ( \begin{array}{c} a , -b \\ c \end{array}; 1 \right ) = \frac{(c-a)_b}{(c)_b},
\end{equation}
which is known as the {\em Chu-Vandermonde sum} \cite[Corollary 2.2.3]{Andrews99SpecialFunctions}; and
\begin{equation}
\label{eq:pfaff}
_{3}F_{2} \left ( \begin{array}{c} a, b, -c \\ d, 1 + a + b - c - d \end{array}; 1 \right ) = \frac{(d-a)_c (d-b)_c}{(d)_c (d-a-b)_c},
\end{equation}
which is referred to as the {\em Pfaff-Saalsch\"utz sum} \cite[Theorem 2.2.6]{Andrews99SpecialFunctions}.
We will use these identities in Section~\ref{sec:neighborly} to prove Theorem~\ref{thm:dimOdd}.

\section{Tightness and polyhedral critical point theory}
\label{sec:critPTheory}

Slicings and rsl-functions as explained in Section~\ref{ssec:rsl} are linked to the concept of tightness via 
Definition~\ref{def:tight3} and the following observation.
%

\begin{prop}
	\label{prop:injHom}
	Let $M$ be a combinatorial $d$-manifold which is tight with respect to a field $\mathbb{F}$, 
	and let $S(g,v) \subset M$ be a slicing. Then there exist an injective homomorphism
	$$ H_i (S(g,v),\mathbb{F}) \to H_i(M,\mathbb{F}) \oplus H_{i+1} (M,\mathbb{F}). $$
\end{prop}

\begin{proof}
	Let $S(g,v) \subset M$ be a slicing of a tight $d$-manifold $M$
	and let $c_i \in H_i (S(g,v),\mathbb{F})$. Then either $c_i$ is an element of both
	$H_{i} (M^+,\mathbb{F})$ and $H_{i} (M^-,\mathbb{F})$; or $c_i$ is the boundary of
	a sub-complex in both $M^+$ and $M^-$, and both sub-complexes glued along their common boundary $c_i$ 
	become an element of $H_{i+1} (M,\mathbb{F})$.
\end{proof}

In particular it follows from Proposition~\ref{prop:injHom} that 
$\beta_i (S,\mathbb{F}) \leq \beta_i (M,\mathbb{F}) + \beta_{i+1} (M,\mathbb{F})$.
Even more, we can improve Proposition~\ref{prop:injHom} in the following way.

\begin{prop}
	\label{prop:beta}
	Let $M$ be a $d$-dimensional combinatorial manifold with
	set of vertices $V$ which is tight with respect to $\mathbb{F}$. Then we have
	$$ \underset{g \textrm{ rsl-function}}{\max} \,\, \sum \limits_{i=0}^{d-1} \,\,
	\underset{1 \leq j \leq n}{\max} \,\, \beta_i (S(g,j),\mathbb{F}) \quad \leq \quad \sum \limits_{i=0}^{d} \beta_i (M,\mathbb{F}). $$
\end{prop}

\begin{proof}
	First of all note that it is essential to first pick an rsl-function $g$ and then 
	look at the maximum Betti number of $S(g,j)$ for all levels $j$ of $g$.
	In particular the statement is false if we sum over the maximum Betti numbers
	of slicings associated to distinct rsl-functions.
\end{proof}

So far when talking about tightness we did not pay attention to the choice of field $\mathbb{F}$.
However, if we choose $\mathbb{F}$ carefully the problem of determining whether $M$ is tight 
becomes more clear. 

\begin{prop}[Definition 2.7 \cite{Bagchi14StellSpheresTightness}]
	Let $M$ be a combinatorial $d$-manifold. Then the following are equivalent:
	\begin{enumerate}[(i)]
		\item $M$ is tight with respect to all fields,
		\item $M$ is tight with respect to all fields of prime order,
		\item $M$ is tight with respect to $\mathbb{Q}$.
	\end{enumerate}
\end{prop}

Furthermore, we have

\begin{prop}[Proposition 2.9(b) \cite{Bagchi14StellSpheresTightness}]
	Let $M$ be a combinatorial $d$-manifold which is tight with respect to a field $\mathbb{F}$. 
	Then $M$ is also $\mathbb{F}$-orientable.
\end{prop}

Thus, {\em (i)} any non-orientable manifold is at most $\mathbb{F}_2$-tight
and when talking about tightness we 
can always assume Poincar\'e duality holds, and {\em (ii)} any $\mathbb{Q}$-tight
combinatorial manifold is also $\mathbb{F}_2$-tight. Furthermore, most 
combinatorial manifolds cannot be tight unless they are tight with respect to $\mathbb{F}_2$ 
and there is no example of a tight combinatorial manifold known to the author
for which this statement doesn't hold. For this reason we will treat $\mathbb{F}_2$-tightness 
as equivalent to tightness whenever possible and will only
consider other fields when necessary. 

In particular, whenever we talk about critical points and Betti numbers
of a tight combinatorial manifold $M$ we will omit the field $\mathbb{F}$ and simply
assume that either $\mathbb{F} = \mathbb{F}_2$ or $\mathbb{F}$ is chosen such that $M$ 
is tight with respect to $\mathbb{F}$.

These simple observations allow us to impose bounds on the Betti numbers of tight combinatorial 
manifolds by looking at the topology of slicings. Note that in the case of odd-dimensional combinatorial
manifolds this leaves us with studying even-dimensional slicings inside the manifold. 
This is of particular convenience because, as already pointed out in 
\cite{Kuehnel95TightPolySubm}, questions about tightness are much harder to investigate 
in the odd-dimensional case due to the vanishing Euler characteristic.
Looking at slicings enables us to use Euler characteristic arguments in the
odd-dimensional case as well. This will complement observations 
made by Effenberger who gave a tightness criterion
in the odd-dimensional case $d \geq 5$ by looking at properties
of the (even-dimensional) vertex links \cite{Effenberger09StackPolyTightTrigMnf}.

In Section~\ref{sec:upperBound} and \ref{sec:neighborly} we present necessary conditions 
for tightness which are most powerful in the missing case of dimension three.
However, let us first state a further {\em topological} constraint for tightness dimension three.

\begin{prop}
	\label{prop:heegaard}
	Let $\mathbb{M}$ be an $\mathbb{F}$-orientable three-manifold of Heegaard genus $\mathbf{g}$
	such that $\beta_1 (\mathbb{M},\mathbb{F}) < \mathbf{g}$, then $\mathbb{M}$ 
	cannot have an $\mathbb{F}$-tight triangulation, that is, there is no 
	$\mathbb{F}$-tight combinatorial manifold $M \cong \mathbb{M}$.
\end{prop}

\begin{proof}
	Assume that $M$ is a combinatorial manifold homeomorphic to $\mathbb{M}$ with vertex
	set $V(M) = \{ v_1, \ldots , v_n \}$ which is tight with respect to $\mathbb{F}$. Then every rsl-function 
	$g : M \to \mathbb{R}$ has exactly $\beta_1 (M,\mathbb{F})$ critical points of index one.  

	We will prove the statement by using $g$ to construct a handlebody decomposition of $M$
	of genus $\beta_1 (M,\mathbb{F})$. This will then contradict the assumption
	$\beta_1 (M,\mathbb{F}) < \mathbf{g}$.

	\medskip
	Without loss of generality, let $g : M \to \mathbb{R}$ be given by $g (v_1) < g(v_2) < \ldots < g(v_n)$.
	Start by setting $H$ to be a neighbourhood around $v_1$, the unique critical point of index zero. 
	Now proceed by considering every vertex in the ordering given by $g$. If $v_i$ is not critical of
	index one, unite $H$ with a neighbourhood of the edge going from $v_{i-1}$ to $v_i$ (note that $M$ being
	tight implies that $M$ is two-neighbourly and thus every two points are connected by an edge).
	If $v_i$ is critical of index one and multiplicity $m$, the intersection of the link of $v_i$ and
	$H$ will have precisely $m+1$ connected components, for each of them unite $H$ with a small neighbourhood 
	of an edge connecting it to $v_i$. As a result we will get a handle body $H \subset M$ of genus $\beta_1 (M,\mathbb{F})$.

	To see that $M \setminus H$ is also a handlebody consider the rsl-function $-g : M \to \mathbb{R}$ restricted to $M \setminus H$.
	By construction $-g$ will not have any critical points of index two or three in $M \setminus H$ and only one critical point of index zero
	and $\beta_1 (M,\mathbb{F})$ critical points of index one. Thus, $M \setminus H$ is a handle body and $(H, M \setminus H)$
	is a handle body decomposition of $M$ of genus $\beta_1 (\mathbb{M},\mathbb{F})$.
	Contradiction to the assumption that $M$ has Heegaard genus
	$\beta_1 (\mathbb{M},\mathbb{F}) < \mathbf{g}$. Hence, $\mathbb{M}$ does not admit any $\mathbb{F}$-tight triangulation. 
\end{proof}

The construction given above is a discrete version of a standard technique in smooth Morse theory.
Given a $d$-manifold $M$ and a Morse function $f : M \to \mathbb{R}$, a handle decomposition $H$
(note that Heegaard decompositions are a special case of handle decompositions) 
can be constructed by adding an $i$-handle for each critical point of index $i$
(see for instance \cite[Section 4.2]{Gompf}).
Thus, replacing the Heegaard genus by the size of a handle decomposition with the minimum
number of handles might help to generalise this statement to the higher dimensional case.

\section{An upper bound for tight odd-dimensional combinatorial manifolds}
\label{sec:upperBound}

In this section we will look at the following situation.
Let $M$ be an $n$-vertex combinatorial $(2 \ell+1)$ manifold 
with vertex set $V$ and $f$-vector $f(M) = (n,f_1, \ldots ,f_{2k+1})$.
Furthermore, let 
$$ \mathbb{S}_k := \{ S(g,v) \subset M \,|\, M^- \textrm{ contains } k \textrm{ vertices } \}$$
be the set of slicings of $M$ separating $k$ vertices from 
the other $n-k$ vertices and let
$$ \bar{f} (\mathbb{S}_k) =  (\bar{f}_{0,k},\bar{f}_{1,k} , \ldots , \bar{f}_{2 \ell, k}) $$
be the {\it average $f$-vector of $S \in \mathbb{S}_k$}, that is, 
the sum over all $f$-vectors of slicings in $\mathbb{S}_k$ divided
by the cardinality of $\mathbb{S}_k$. With these definitions we can state

\begin{prop}
	\label{prop:avgf}
	Let $M$ be an $n$-vertex combinatorial $(2\ell +1)$-manifold with vertex set $V$.
	The the average number of $i$-faces of a slicing $S \in \mathbb{S}_k$ equals
	$$ \bar{f}_{i,k} = f_{i+1} \left ( 1 - \frac{{k \choose i+2} + {n-k \choose i+2}}{{n \choose i+2}} \right ) $$
	where $f_{i+1}$ is the number of $(i+1)$-faces of $M$, and
	the average Euler characteristic of slicings in $\mathbb{S}_k$ is given by
	$$ \bar{\chi}_k = \sum \limits_{i=0}^{2 \ell+1} (-1)^{i} \frac{f_{i}}{{n \choose i+1}} \left [ {k \choose i+1} + {n-k \choose i+1} \right ] .$$
	In particular, $\bar{\chi}_k$ does not depend on topological properties of $M$.
\end{prop}

\begin{proof}
	Fix a $k$ vertex subset $\Delta \subset V$ of $M$ 
	defining a slicing $S \in \mathbb{S}_k$ and consider 
	subsets $\delta \subset V$ of size $(i+2)$.
	$\delta$ intersects with $S$ if and only 
	if it is neither disjoint to nor contained in $\Delta$.

	Now, the number of ways exactly $j$ vertices, $0 < j < i+2$,
	of $\delta$ lie in $\Delta$ is ${ k \choose j }$ and
	for each of these choices, there are exactly ${ n-k \choose i+2-j }$
	ways for the other $i+2-j$ vertices of $\delta$ to be chosen of
	the remaining $n-k$ vertices. Summing up over $j$ this leaves us with 
	$$ \sum \limits_{j=1}^{i+1} {k \choose j} {n-k \choose i+2-j} $$
	possible intersections of subsets $\delta$ with the slicing given
	by $\Delta$. Now, if $\delta$ is an $(i+1)$-face of $M$ each 
	of these intersections results in exactly one face of $S$, and
	since there are ${n \choose i+2}$ such subsets
	$\delta$ but only $f_{i+1}$ $(i+1)$-faces we have for the average
	number of $(i+1)$-faces of $S$
	\begin{equation}
		\label{eq:fbar}
		\bar{f}_{i,k} = \frac{f_{i+1}}{{n \choose i+2}} \sum \limits_{j=1}^{i+1} {k \choose j}{n-k \choose i+2-j} .
	\end{equation}

	\medskip
	Applying the Chu--Vandermonde sum (cf. Equation~(\ref{eq:chu})) we get that
	\begin{equation}
		\label{eq:sum}
		\sum \limits_{j=1}^{i+1} {k \choose j} {n-k \choose i+2-j} = {n \choose i+2} - \left ({k \choose i+2} + {n-k \choose i+2} \right )
	\end{equation}
	and the result follows by introducing the RHS of Equation~(\ref{eq:sum}) into Equation~(\ref{eq:fbar}) for $\bar{f}_{i,k}$. Hence, we have
	%
	\begin{align*}
	\bar{\chi}_k &\quad=\quad \sum \limits_{i=0}^{2 \ell} (-1)^{i} \bar{f}_{i,k} \\
	&\quad=\quad \sum \limits_{i=0}^{2 \ell} (-1)^{i} f_{i+1} -  \sum \limits_{i=0}^{2 \ell} (-1)^i \frac{f_{i+1}}{{n \choose i+2}} \left [ {k \choose i+2} + {n-k \choose i+2} \right ] \\
	&\quad=\quad n + \sum \limits_{i=0}^{2 \ell} (-1)^{i+1} \frac{f_{i+1}}{{n \choose i+2}} \left [ {k \choose i+2} + {n-k \choose i+2} \right ] \\
	&\quad=\quad \sum \limits_{i=0}^{2 \ell+1} (-1)^{i} \frac{f_{i}}{{n \choose i+1}} \left [ {k \choose i+1} + {n-k \choose i+1} \right ] 
	\end{align*}
	which proves the statement.
\end{proof}

Very recently, Swartz proved a similar but more general result grouping together discrete normal surfaces 
of a combinatorial manifold $M$ where their corresponding dual one-cocycles lie in the same co-homology class
\cite{Swartz13AvgChi}. Slicings are precisely those discrete normal surfaces where the dual one-cocycle 
is trivial in co-homology. Proposition~\ref{prop:avgf} is thus a special case of Lemma 2.1 in \cite{Swartz13AvgChi}.

The significance of Proposition~\ref{prop:avgf} and the results in \cite{Swartz13AvgChi}
is that they allow to compute the average Euler characteristic
of slicings {\em independently} of the topology of the 
manifold $M$ itself. The average Euler characteristic only depends
on the number of faces of $M$ in each dimension. A priori this information
doesn't reveal any topological features of odd-dimensional combinatorial manifolds (e.g., by a theorem
of Sarkaria and Walkup \cite{Sarkaria83OnNeighTrig,Walkup70LBC34Mnf}, 
any three-manifold admits a triangulation with $f$-vector $f = (n,{n \choose 2},2({n \choose 2}-n),{n \choose 2}-n)$
for $n$ sufficiently large).
However, by Proposition~\ref{prop:avgf} almost all $f$-vectors of odd-dimensional 
combinatorial manifolds will give rise to either a strictly negative or a strictly 
positive average Euler characteristic. This postulates the existence of a 
slicing with non-trivial topology imposing lower bounds on the odd-dimensional Betti numbers 
of the surrounding manifold in the former, and lower bounds on the even-dimensional Betti numbers
in the latter case. By Proposition~\ref{prop:beta} these bounds then directly translate 
to topological constraints for {\em tight} combinatorial manifolds.

\medskip
In other words, Proposition~\ref{prop:avgf} confirms the intuition one might have that a manifold 
where some of its topological features are not visible in homology has a small chance of admitting 
a tight triangulation at all (cf. Proposition~\ref{prop:heegaard}). 

In the next Section we will investigate how Proposition~\ref{prop:avgf} can be restated
in the important special case of tight $(\ell-1)$-connected $(2\ell+1)$-manifolds.

\section{$(\ell-1)$-connected $(2\ell+1)$-manifolds}
\label{sec:neighborly}

In \cite{Kuehnel95TightPolySubm} K\"uhnel gives a tightness criterion for $(\ell-1)$-connected $2\ell$-manifolds.
This result is complemented by a classification of tight combinatorial three-manifolds with first Betti number at most one and
upper bounds on the number of vertices for 
simply connected five-manifolds with first Betti number at most one \cite[Theorem 5.3 and Proposition 7.3]{Kuehnel95TightPolySubm}.
Here, we want to generalise the latter set of results to the case of $(\ell-1)$-connected $(2\ell +1)$-manifolds
(note that K\"uhnel's results for odd dimensions can be recovered from Theorem~\ref{thm:dimOdd} as a special case).

Any tight $(\ell-1)$-connected combinatorial $(2\ell+1)$-manifold must be $(\ell +2)$-neighbourly. To see this
assume that there is a combinatorial $(2\ell +1)$-manifold $M$
with a minimal missing $k$-face $\Delta$, i.e., all proper subsets of 
$\Delta$ span a face in $M$ while $\Delta$ is not a face of $M$, $k \leq \ell+1$. 
Then the slicing separating the span of $\Delta$
and the rest of the manifold gives an obstruction to tightness.
Now by the upper bound theorem 
\cite{Novik98UBTHomMnf,Novik08SocBuchsMod} we know that an 
$(\ell +2)$-neighbourly combinatorial $(2\ell +1)$-manifold $M$ must have the 
$f$-vector of the boundary complex of the cyclic $(2\ell +2)$-polytope
and is thus determined by fixing the dimension and the number of vertices $n$.

Before we can prove Theorem~\ref{thm:dimOdd}, we first have to establish
some useful tools.

\begin{lemma}
	\label{lem:sums}
	Let
	$$ s_{i,j} (k,n) = (-1)^i {k\choose i}{n-1 \choose i}^{-1} {n-j-1 \choose i-j}{n-i \choose 2j -i}.$$
	Then for any $j \geq 0$ we have 
	\begin{equation}\label{eq:i}
		\sum_{i \geq 0} s_{i,j}(k,n)=(-1)^j{k \choose j}{n-k \choose j}{n-1 \choose j}^{-1}
	\end{equation}
	and for any integer $i$
	\begin{equation}\label{secondsum}
		\sum_{j \geq 0} s_{i,j}(k,n)=(-1)^i \binom{k}{i}
	\end{equation}
	holds.
\end{lemma}

\begin{proof}
	First note that both Equation~(\ref{eq:i}) and (\ref{secondsum}) are finite sums since $s_{ij}(k,n) = 0$
	whenever $j > i$ or $i > 2j$.
	
	To show Equation~\ref{eq:i} first perform an index shift followed by a re-arrangment of factorials resulting in
	\normalsize
	\begin{align*}
		\sum \limits_{i \geq 0} s_{i,j}(k,n) &\quad\stackrel{i \mapsto i+j}{=}\quad \sum \limits_{i=0}^{j} s_{i+j,j} \\
		&\quad=\quad (-1)^{j} {k \choose j} {n-1 \choose j}^{-1} {n-j \choose j}  
			{_2}F_1\left(\genfrac{}{}{0pt}{}{j-k,-j}{j-n}{};1\right) . \\ 
	\end{align*}
	\normalsize
	Now, by the Chu--Vandermonde sum (see Equation~(\ref{eq:chu}) or \cite[Corollary 2.2.3]{Andrews99SpecialFunctions})
	$$ {_2}F_1\left (\genfrac{}{}{0pt}{}{a,-n}{b};1\right ) =\frac{(b-a)_n}{(b)_n} $$
	the statement follows.
%
	\medskip
	As for Equation~(\ref{secondsum}) we have

	\begin{align*}
	\sum \limits_{j \geq 0} s_{i,j}(k,n) \quad&\stackrel{j\mapsto i-j}{=} \sum \limits_{j =0}^{\lfloor i/2 \rfloor} s_{i,i-j}(k,n) \\
	&\quad=\quad (-1)^{i} {n-1 \choose i}^{-1} {k \choose i} {n-i \choose i} 
			\sum \limits_{j \in \mathbb{Z}} \frac{(n-i)_j (-i)_{2j}}{j! (n-2i+1)_{2j}} \\
	&\quad=\quad (-1)^{i} {n-1 \choose i}^{-1} {k \choose i} {n-i \choose i}  \left \{
			\begin{array}{ll} 
				{_3}F_2\left (\genfrac{}{}{0pt}{}{n-i,-i/2,-(i-1)/2}{(1+n-2i)/2,(2+n-2i)/2}{};1\right ) 
				& \textrm{ if } i \equiv 1(2) \\
				{_3}F_2\left (\genfrac{}{}{0pt}{}{n-i,-(i-1)/2,-i/2}{(1+n-2i)/2,(2+n-2i)/2}{};1\right ) 
				& \textrm{ otherwise} \\
			\end{array} \right .
	\end{align*}
	where the last step follows from Equation~(\ref{eq:2}).

	Hence, we have a hypergeometric sum satisfying the pre-conditions of the Pfaff--Saalsch\"utz sum (see Equation~(\ref{eq:pfaff})
	or \cite[Theorem 2.2.6]{Andrews99SpecialFunctions}) stating that
	\begin{equation}\label{Saal}
		{_3}F_2\left(\genfrac{}{}{0pt}{}{a,\,\,\,b,\,\,\,-n}{d,1+a+b-n-d}{};1\right) =\frac{(d-a)_n(d-b)_n}{(d)_n(d-a-b)_n}
	\end{equation}
	whenever $n$ is a non-negative integer.

	\medskip	
	\noindent
	Now if $i$ is odd we set $n = (i-1)/2$, and when $i$ is even we set $n = i/2$ and by re-arranging factorials we can see that
	most terms cancel out and Equation~(\ref{secondsum}) follows.
\end{proof}

With these identities in mind, we can now prove the following.

\begin{lemma}
	\label{lem:fvecCycPol}
	The cyclic $(2 \ell +2)$-polytope with $n$ vertices $C_{2m} (n)$ has 
	\begin{equation}
		\label{eq:fvec}
		f_{i-1} (C_{2m} (n)) = \frac{n}{n-i} \sum \limits_{j = 0}^{\ell + 1} {n - j - 1 \choose i-j} {n-i \choose 2j -i}
	\end{equation}
	faces of dimension $(i-1)$.
\end{lemma}

\begin{proof}
	If $i > \ell + 1$ this follows directly from \cite[Theorem 15.3.4]{Billera97FaceNumPolyCompl}.

	\medskip
	Now let $i \leq \ell + 1$, that is, we have to show that
	$$ \frac{n}{n-i} \sum \limits_{j = 0}^{\ell + 1} {n - j - 1 \choose i-j} {n-i \choose 2j -i} = {n \choose i}.  $$
	First note that the summands in Equation~(\ref{eq:fvec}) vanish unless
	$0 \leq j \leq i \leq 2j$. Thus, if $i \leq \ell + 1$ we have
	$$ 
		\sum \limits_{j = 0}^{\ell + 1} {n - j - 1 \choose i-j} {n-i \choose 2j -i} 
		= \sum \limits_{j \in \mathbb{Z}} {n - j - 1 \choose i-j} {n-i \choose 2j -i}.
	$$
	In addition, Lemma~\ref{lem:sums} states that
	$$\sum \limits_{j \in \mathbb{Z}} (-1)^i {k\choose i}{n-1 \choose i}^{-1} {n-j-1 \choose i-j}{n-i \choose 2j -i} = (-1)^i  {k\choose i}$$
	which is equivalent to
	\begin{equation}
		\label{eq:sumi}
		\sum \limits_{j \in \mathbb{Z}} {n-j-1 \choose i-j}{n-i \choose 2j -i} = {n-1 \choose i}
	\end{equation}
	and replacing the sum in Equation~(\ref{eq:fvec}) with the RHS of Equation~(\ref{eq:sumi}) we obtain the result.
\end{proof}

Now note that for $\mathbb{F}$-orientable $(\ell-1)$-connected $(2\ell +1)$-manifolds $M$
we have $\beta_1 (M,\mathbb{F}) = \beta_2 (M,\mathbb{F}) = \ldots = \beta_{\ell-1} (M\mathbb{F}) = 0$.
Furthermore, by Proposition~\ref{prop:injHom} we know
that $\beta_{\ell} (S,\mathbb{F}) \leq \beta_{\ell} (M\mathbb{F}) + \beta_{\ell+1} (M,\mathbb{F}) = 2 \beta_{\ell} (M,\mathbb{F})$
and for any slicing $S \subset M$ we have $\chi (S) \geq 2 - \beta_{\ell} (S,\mathbb{F})$ if $\ell$ 
is odd and $\chi (S) \leq 2 + \beta_{\ell} (S,\mathbb{F})$ if $\ell$ is even. 

All together this results in
	$$ \beta_{\ell} (M,\mathbb{F}) = \beta_{\ell+1} (M,\mathbb{F}) \geq 
		\left \{ \begin{array}{ll} \chi (S) / 2 - 1 & \textrm{ if } \ell \textrm{ is even } \\  
		1 - \chi (S) / 2 & \textrm{ else }  \end{array} \right . $$
and thus the inequality
	\begin{align*} 
		\beta_{\ell} (M,\mathbb{F}) = \beta_{\ell+1} (M,\mathbb{F}) \geq 
			(-1)^{\ell+1} \left ( 1- \bar{\chi}_{k}/2 \right ) &\quad=\quad (-1)^{\ell+1} 
			\left ( 1 - \frac12 \left ( \sum \limits_{i=0}^{2 \ell+1} (-1)^{i} \frac{f_{i}}{{n \choose i+1}} 
			\left [ {k \choose i+1} + {n-k \choose i+1} \right ] \right ) \right )\\
		&\quad=\quad \frac{(-1)^{\ell+1}}{2} \sum \limits_{i=0}^{2 \ell+2} 
			\sum \limits_{j=0}^{\ell+1} (-1)^{i} \frac{{n-1-j \choose i-j} {n - i \choose 2j - i}}{{n-1 \choose i}} 
			\left [ {k \choose i} + {n-k \choose i} \right ] \\
		&\quad=\quad \frac{(-1)^{\ell+1}}{2} \sum \limits_{i=0}^{2 \ell+2} 
			\sum \limits_{j=0}^{\ell+1} \left [ s_{i,j}(k,n) + s_{i,j}(n-k,n) \right ] , \\
	\end{align*}
with $s(k,n)$ as defined in Lemma~\ref{lem:sums}, must hold for all $k \leq n/2$.

\begin{proof}[Proof of Theorem~\ref{thm:dimOdd}]
	By the above calculations and Lemma~\ref{lem:fvecCycPol} it suffices to show that
	$$ 
		\sum \limits_{i=0}^{2 \ell+2} \sum \limits_{j=0}^{\ell+1} s_{i,j}(k,n) = 
		\frac{ (1-k)_{\ell+1} (1-n+k)_{\ell+1} }{ (\ell+1)! \,\, (1-n)_{\ell+1}} 
	$$
	for all $k \leq n/2$. Applying Equation~(\ref{eq:i}) 
	from Lemma~\ref{lem:sums} and letting $c$ tend to $-n$ in the 
	Pfaff--Saalsch\"utz sum Equation~(\ref{Saal}) we have
	\begin{align*}
		\sum \limits_{i=0}^{2 \ell+2} \sum \limits_{j=0}^{\ell+1} s_{i,j}(k,n) &\quad=\quad 
			\sum \limits_{j \geq 0} \sum \limits_{i \geq 0} s_{i,j}(k,n) \\
		&\quad=\quad \sum \limits_{j \geq 0} (-1)^j {k \choose j}{n-k \choose j}{n-1 \choose j}^{-1} \\
		&\quad=\quad \frac{(1-k)_{\ell+1} (1-n+k)_{\ell+1}}{(\ell+1)! (1-n)_{\ell+1}}.
	\end{align*}

	Note that the above statement is most restrictive in the case $k = \lfloor n/2 \rfloor$. 
	Finally, multiplying the equation by $-1$ whenever $\ell$ is even completes the proof.
\end{proof}

In dimensions $d=2\ell + 1$, $1 \leq \ell \leq 15$, 
Theorem~\ref{thm:dimOdd} translates to the following upper bounds 
for vertex numbers of tight $(\ell-1)$-connected $d$-manifolds.

\medskip
\begin{center}
  \begin{tabular}{|c|@{\hspace{2.5mm}}c@{\hspace{2.5mm}}|@{\hspace{1mm}}c@{\hspace{1mm}}|@{\hspace{1mm}}c@{\hspace{1mm}}|@{\hspace{1mm}}c@{\hspace{1mm}}|@{\hspace{1mm}}c@{\hspace{1mm}}|@{\hspace{1mm}}c@{\hspace{1mm}}|@{\hspace{1mm}}c@{\hspace{1mm}}|@{\hspace{1mm}}c@{\hspace{1mm}}|@{\hspace{1mm}}c@{\hspace{1mm}}|@{\hspace{1mm}}c@{\hspace{1mm}}|@{\hspace{1mm}}c@{\hspace{1mm}}|@{\hspace{1mm}}c@{\hspace{1mm}}|@{\hspace{1mm}}c@{\hspace{1mm}}|@{\hspace{1mm}}c@{\hspace{1mm}}|@{\hspace{1mm}}c@{\hspace{1mm}}|}
	\hline
	\backslashbox{$\beta_{\ell} (M)$}{$d$}&$3$&$5$&$7$&$9$&$11$&$13$&$15$&$17$&$19$&$21$&$23$&$25$&$27$&$29$&$31$ \\
	\hline \hline
	$0$&{\bf5}&{\bf7}&{\bf9}&{\bf11}&{\bf13}&{\bf15}&{\bf17}&{\bf19}&{\bf21}&{\bf23}&{\bf25}&{\bf27}&{\bf29}&{\bf31}&{\bf33} \\
	\hline
	$1$&$10$&{\bf 13}&$16$&$19$&$22$&$25$&$28$&$31$&$34$&$37$&$40$&$43$&$46$&$49$&$52$ \\
	\hline
	$2$&$12$&$15$&$17$&$20$&$23$&$26$&$29$&$32$&$35$&$39$&$41$&$45$&$48$&$51$&$54$ \\
	\hline
	$3$&$14$&$16$&$19$&$21$&$24$&$27$&$30$&$33$&$36$&--\textquotedbl--&$42$&--\textquotedbl--&--\textquotedbl--&--\textquotedbl--&--\textquotedbl-- \\
	\hline
	$4$&$15$&$17$&--\textquotedbl--&$22$&$25$&$28$&$31$&$34$&$37$&$40$&$43$&$46$&$49$&$52$&$55$ \\
	\hline
	$5$&$17$&$18$&$20$&$23$&$26$&$29$&--\textquotedbl--&--\textquotedbl--&--\textquotedbl--&--\textquotedbl--&--\textquotedbl--&--\textquotedbl--&--\textquotedbl--&--\textquotedbl--&--\textquotedbl-- \\
	\hline
	$6$&$18$&$19$&$21$&--\textquotedbl--&--\textquotedbl--&--\textquotedbl--&$32$&$35$&$38$&$41$&$44$&$47$&$50$&$53$&$56$ \\
	\hline
	$7$&$19$&--\textquotedbl--&--\textquotedbl--&$24$&$27$&--\textquotedbl--&--\textquotedbl--&--\textquotedbl--&--\textquotedbl--&--\textquotedbl--&--\textquotedbl--&--\textquotedbl--&--\textquotedbl--&--\textquotedbl--&--\textquotedbl-- \\
	\hline
	$8$&$20$&$20$&$22$&--\textquotedbl--&--\textquotedbl--&$30$&$33$&--\textquotedbl--&--\textquotedbl--&--\textquotedbl--&--\textquotedbl--&--\textquotedbl--&--\textquotedbl--&--\textquotedbl--&--\textquotedbl-- \\
	\hline
	$9$&$21$&$21$&--\textquotedbl--&$25$&--\textquotedbl--&--\textquotedbl--&--\textquotedbl--&$36$&$39$&$42$&--\textquotedbl--&--\textquotedbl--&--\textquotedbl--&--\textquotedbl--&--\textquotedbl-- \\
	\hline
	$10$&$22$&--\textquotedbl--&$23$&--\textquotedbl--&$28$&--\textquotedbl--&--\textquotedbl--&--\textquotedbl--&--\textquotedbl--&--\textquotedbl--&$45$&$48$&$51$&$54$&$57$ \\
	\hline
  \end{tabular}
\end{center}

\noindent
--\textquotedbl-- denotes the same value as above.

\noindent
Bold entries denote cases where tight triangulations exist.

\medskip
In particular, the bound of Theorem~\ref{thm:dimOdd} is attained for $d=5$ and $\beta_2(M) = 1$
as there is a tight $13$-vertex triangulation of $SU(3) / SO(3)$ described in \cite{Kuehnel99CensusTight}.

\begin{remark}
	$(\ell-1)$-connected $(2\ell +1)$-manifolds are less 
	restrictive than an $(\ell-1)$-connected $2\ell$-manifold. This
	is most apparent in the case of $\ell=1$ where 
	Theorem~\ref{thm:dimOdd} holds for all connected 
	three-manifolds (cf. Section~\ref{sec:sigma} where we will deal
	with the three-dimensional case in more detail). This is of particular interest as dimension
	three is the unique non-trivial odd dimension where Effenberger's
	tightness criterion \cite{Effenberger09StackPolyTightTrigMnf} 
	cannot be applied.
\end{remark}

Following Theorem~\ref{thm:dimOdd} a tight $n$-vertex $(\ell -1)$ connected 
$(2 \ell + 1)$-manifold $M$ must satisfy
\begin{equation}
	\label{eq:myCriterion}
	\beta_{\ell} (M,\mathbb{F}) \geq \frac{n^{\ell+1}}{4^{\ell+1} (\ell+1)!} + o (n^{\ell + 1}) .
\end{equation}
For three-manifolds we thus have $\beta_{1} (M,\mathbb{F}) \geq \frac{n^2}{32} + o(n^2)$. On the other hand it follows
from \cite{Novik08SocBuchsMod} and \cite[Theorem 5]{Lutz08FVec3Mnf} 
that $\beta_{1} (M,\mathbb{F}) \leq \frac{n^2}{20} + o(n^2)$ for all combinatorial
three-manifolds. Thus, as $n$ tends to infinity, the number of vertices of a tight combinatorial three-manifold $M$ cannot be greater
than $\sqrt{1.6} \simeq 1.265$ times the minimum number of vertices needed to triangulate
any manifold $N$ with first Betti number $\beta_1 (M,\mathbb{F})$. In particular, the conjecture by Lutz and K\"uhnel
stating that tight combinatorial manifolds are strongly minimal \cite[Conjecture 1.3]{Kuehnel99CensusTight} cannot be far from 
being true in dimension three.


\begin{remark}[Comparison with Bagchi and Datta's tightness criterion]
In \cite{Bagchi14StellSpheresTightness} the authors give a tightness 
criterion for all combinatorial manifolds in Walkup's class
$\mathcal{W}^{\star}_k(d)$. In the $(\ell-1)$-dimensional $(2 \ell + 1)$-dimensional case these $n$-vertex combinatorial
manifolds must satisfy 
$$\beta_{\ell} (M,\mathbb{F}) = \frac{{n-\ell -3 \choose \ell + 1}}{{2\ell + 3 \choose \ell + 1}} 
\simeq \frac{(\ell + 2)! \,\, n^{\ell + 1}}{(2 \ell + 3)!} + o (n^{\ell}). $$
Comparing this to Equation~(\ref{eq:myCriterion}) we get that, for $n$ tending to infinity, any tight combinatorial $(\ell -1)$-connected
$(2 \ell + 1)$-manifold cannot have more than
$$ \sqrt[\ell+1]{\frac{4^{\ell + 1} (\ell +1)! (\ell + 2)!}{(2 \ell + 3)!}} \quad=\quad 4 \cdot \exp^{\frac{1}{l+1} \left [\ln (\ell+1)! + \ln (\ell+2)! - \ln (2 \ell+3)! \right ]} $$
times the vertices than any tight combinatorial $(\ell -1)$-connected
$(2 \ell + 1)$-manifold in $\mathcal{W}^{\star}_k(d)$. And since
\begin{align*}
	\ln (\ell+1)! + \ln (\ell+2)! - \ln (2 \ell+3)! &\quad=\quad \sum \limits_{x=1}^{\ell+1} \ln (x) + \sum \limits_{x=1}^{\ell+2} \ln (x) - \sum \limits_{x=1}^{2\ell+3} \ln (x) \\
	&\quad=\quad \sum \limits_{x=1}^{\ell+2} \ln (x) - \sum \limits_{x=1}^{\ell+1} \ln (x + 1/2) - (\ell + 1) \ln (4) \\
	&\quad<\quad \ln (\ell+2) - (\ell + 1) \ln (4),
\end{align*}
we get 
$$ \sqrt[\ell+1]{\frac{4^{\ell + 1} (\ell +1)! (\ell + 2)!}{(2 \ell + 3)!}} \quad \overset{\ell \to \infty}{\longrightarrow} \quad 1 .$$

Hence, Theorem~\ref{thm:dimOdd} gives a necessary condition for the tightness of arbitrary combinatorial
$(\ell -1)$-connected $(2 \ell + 1)$-manifolds which asymptotically (in $n$ then $\ell$) tends to Bagchi and Datta's 
tightness criterion for triangulations in $\mathcal{W}^{\star}_k(d)$.
\end{remark}

\section{Three-manifolds with small first Betti number}
\label{sec:sigma}

In this section we will focus on combinatorial three-manifolds
with low first Betti numbers, proving that the only tight combinatorial
three-manifolds with first Betti number at most two are the five-vertex 
boundary of the four-simplex and the nine-vertex triangulation
of the three-dimensional Klein bottle $S^2 \dtimes S^1$ \cite{Kuehnel95TightPolySubm}.

\medskip
In \cite{Bagchi14StellSpheresTightness} a tightness criterion for combinatorial
$d$-manifolds $M$ is given based on calculations in the vertex links
of $M$. More precisely, a condition on the {\em $\sigma$-vectors}
of the vertex links is given for $M$ to be tight. These
are defined as follows.

\begin{definition}[$\sigma$-vector]
	Let $M$ be a combinatorial $d$-manifold with vertex set $V(M) = \{ v_1 , \ldots , v_n \}$.
	The $\sigma$-vector $\sigma(M) = (\sigma_0 , \ldots \sigma_d)$ of $M$ is defined by
	$$ \sigma_i \quad = \quad \sum \limits_{j=0}^{n} \frac{1}{{n \choose j}} \sum \limits_{\footnotesize \begin{array}{c} W \subset V(M), \\ |W| = j \end{array} \normalsize} \tilde{\beta}_i (M[W]), \qquad 
		0 \leq i \leq d .$$ 
\end{definition}

For $d=3$, tight combinatorial manifolds must be two-neighbourly and
any two-neighbourly combinatorial three-manifold is tight if and only if it
is one-tight \cite[Proposition 3.18]{Kuehnel95TightPolySubm},
\cite[Proposition 2.9(b)]{Bagchi14StellSpheresTightness}. 
Thus, in dimension three we have the following tightness 
criterion.

\begin{prop}[Corollary of Theorem 2.10 \cite{Bagchi14StellSpheresTightness}]
	\label{prop:bd}
	Let $M$ be a combinatorial three-manifold with vertex set $V$. Then
	$M$ is $\mathbb{F}$-tight if and only if $M$ is two-neighbourly and
	\begin{equation}\label{eq:bd}
		\frac{1}{|V|} \sum_{v \in V} \sigma_0 (\operatorname{lk}_M(v)) = \beta_1(M,\mathbb{F})-1,
	\end{equation}
	that is, the average value for $\sigma_0$ over all vertex links must equal the first Betti number
	minus one. 
\end{prop}

However, not all combinations of vertex links satisfying Equation~(\ref{eq:bd}) have a chance of being the
set of links of a tight combinatorial three-manifold. Thus, we need to take a closer look at
some essential properties of vertex links of tight combinatorial three-manifolds.

\begin{lemma}[Property $T_k$]
	Let $S$ be an $n$-vertex triangulation of a two-sphere with one-skeleton $G = \operatorname{skel}_1 (S)$
	which occurs as a link of a tight combinatorial three-manifold $M$ with $k = \beta_1 (M,\mathbb{F})$, 
	then the following conditions must be satisfied for at least one fixed field $\mathbb{F}$.
	\begin{enumerate}[(i)]
		\item $k$ satisfies the condition given by Theorem~\ref{thm:dimOdd} for $n+1$ vertices,
		\item $G$ does not have an independent set of size $k+2$, and
		\item $G$ does not have an induced subgraph with six vertices and $k+1$ connected components.
	\end{enumerate}
	If $S$ satisfies all of the above properties it is said to have {\em Property $T_k$}.
\end{lemma}

Note that {\em (ii)} and {\em (iii)} are independent of the choice of $\mathbb{F}$.

\begin{proof}
	For {\em (i)} note that the two-neighbourliness condition implies that any tight combinatorial
	manifold $M$ with $n$-vertex vertex links has $n+1$ vertices. Then apply Theorem~\ref{thm:dimOdd}.

	For {\em (ii)} note that any independent set of size $k+2$ in the one-skeleton $G$ of a vertex link 
	of some vertex $v \in V(M)$ gives rise to an rsl-function with a critical point of index one and 
	multiplicity $k+1$: define $g : M \to \mathbb{R}$ such that $g(w) < g(v)$ for
	all vertices $w$ in the independent set and  $g(v) < g(u)$ for all other vertices. Now 
	$g$ clearly has more than $k$ critical points of index one counted by multiplicity. Hence, $g$
	is not perfect and $M$ cannot be tight.

	For {\em (iii)} let $G$ be the $1$-skeleton of the vertex link of $v \in V(M)$, and let $W \subset V(M) \setminus \{v\}$,
	$W = \{ w_1 , w_2 , \ldots , w_6 \}$, such that $G[W]$ has $k+1$ connected components. Now let $G'$ be
	the $1$-skeleton of the vertex link of $w_1$ in $M$. Since $G'$ is planar, it does not contain a complete
	graph with five vertices and hence there must be a missing edge in the induced subgraph $G' [W \setminus \{w_1\}]$.
	Without loss of generality, let $\{ w_2, w_3\}$ be that missing edge.
	
	With this setup in mind choose an rsl-function of $g : M \to \mathbb{R}$ such that
	$$g (w_2) < g(w_3) < g(w_1) < g(w_4), g(w_5), g(w_6) < g(v) < \ldots .$$
	It follows from the construction that $w_1$ is critical of index one and multiplicity one and 
	$v$ is critical of index one and multiplicity $k$ and $M$ cannot be tight. 
\end{proof}

Now in the three-dimensional case Theorem~\ref{thm:dimOdd} and
\cite[Theorem 5]{Lutz08FVec3Mnf} give us the following upper and lower bounds on 
the vertex numbers of tight combinatorial three-manifolds with prescribed first Betti number.

\medskip
\begin{center}
  \begin{tabular}{|c|c|c|}
	\hline
	$\beta_1 (M,\mathbb{F})$&lower b.&upper b. \\
	\hline \hline
	$0$&{\bf 5}&{\bf 5} \\
	\hline
	$1$&{\bf 9}&$10$ \\
	\hline
	$2$&$11$&$12$ \\
	\hline
	$3$&$13$&$14$ \\
	\hline
	$4$&$14$&$15$ \\
	\hline
	$5$&$15$&$17$ \\
	\hline
	$6$&$16$&$18$ \\
	\hline
	$7$&$17$&$19$ \\
	\hline
	$8$&$18$&$20$ \\
	\hline
	$9$&$18$&$21$ \\
	\hline
	$10$&$19$&$22$ \\
	\hline
	$11$&$20$&$23$ \\
	\hline
	$12$&$20$&$24$ \\
	\hline
  \end{tabular}
\end{center}

\noindent
Bold entries denote cases where tight triangulations exist.

\medskip
In the following we will use these bounds together with Property $T_k$, $k \leq 2$,
the classification of sphere triangulations up to $11$ vertices (triangulations are taken from \cite{Lutz08ManifoldPage}),
as well as the classification of three-manifold triangulations up to $11$ vertices \cite{Sulanke09IsoFreeEnumeration}
to give an alternative classification of tight combinatorial three-manifolds with first Betti number
at most one. Furthermore, we will show that there are no tight combinatorial three-manifolds $M$ with
$\beta_1 (M,\mathbb{F}) =2$ for all fields $\mathbb{F}$.

\subsubsection*{The case $\beta_1 (M,\mathbb{F}) = 0$}

A tight combinatorial three-manifold with vanishing first Betti number must be three-neighbourly (otherwise
consider an rsl-function separating a minimal missing triangle from the rest of the triangulation).
It follows immediately that the only tight combinatorial homology three-sphere is the boundary of the simplex.

\subsubsection*{The case $\beta_1 (M,\mathbb{F}) = 1$}

The case $\beta_1(M,\mathbb{F}) = 1$ is due to K\"uhnel \cite[Theorem 5.3]{Kuehnel95TightPolySubm}.
Alternatively, it can also be followed from Theorem~\ref{thm:dimOdd} and the
classification of combinatorial three-manifolds up to $11$ vertices \cite{Sulanke09IsoFreeEnumeration}.
Here, we will give yet another proof using Theorem~\ref{thm:dimOdd} and
the (much smaller) classification of two-sphere triangulations up to nine vertices.

\medskip
Recall that in a tight $n$-vertex combinatorial three-manifold all vertex links have to be $(n-1)$-vertex 
two-sphere triangulations. Now, following Theorem~\ref{thm:dimOdd}, a tight combinatorial three-manifold 
with $\beta_1 (M,\mathbb{F}) = 1$ needs to have either nine or ten vertices. 

\medskip
\noindent
{\bf Case $n= 9$}: We have $14$ triangulations of the two-sphere with eight vertices
with the following $\sigma_0$-values.

\begin{center}
  \begin{tabular}{|c|c|c|}
	\hline
	$\sigma_0$&$\approx \sigma$&\# two-sphere triangulations \\
	\hline
	\hline
$-2/7$&$-0.2857$&$1$\\
\hline
$-8/35$&$-0.2285$&$1$\\
\hline
$-27/140$&$-0.1928$&$1$\\
\hline
$-9/70$&$-0.1285$&$4$\\
\hline
$0$&$0$&$7$\\
	\hline
	\hline
	& Total:& $14$ \\
	\hline
  \end{tabular}
\end{center}

Hence, to satisfy Equation~(\ref{eq:bd}) only the seven triangulations with $\sigma_0$-value
equal to zero can be considered (these are precisely the seven stacked eight-vertex two-spheres,
see \cite{Burton14SepIndex2Spheres} for a more general observation on the $\sigma_0$-value of 
two-sphere triangulations). Amongst these seven triangulations only the triangulation
presented in Figure~\ref{fig:lktight} satisfies Property $T_1$ and thus any tight combinatorial three-manifold $M$
with $\beta_1 (M,\mathbb{F}) = 1$ must have nine isomorphic vertex links of that type. 
By virtue of the enumeration algorithm presented below, this
leads to the unique nine-vertex triangulation of the three-dimensional Klein bottle.
\begin{figure}
	\begin{center}
		\includegraphics[width=.5\textwidth]{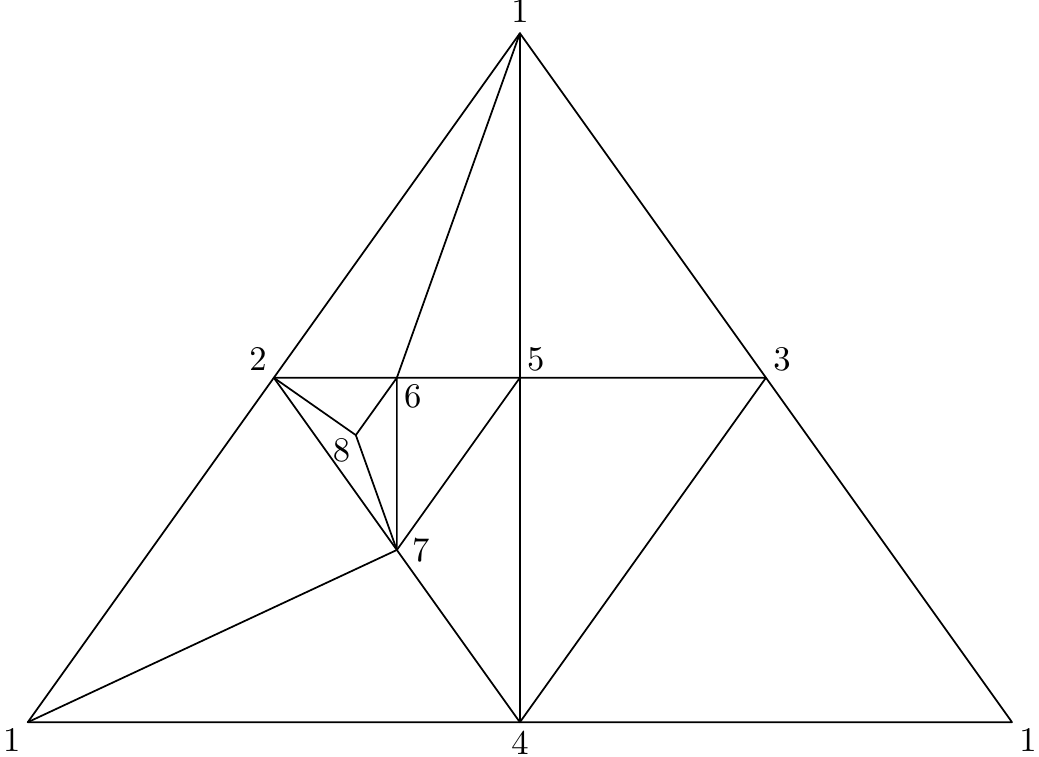}
	\end{center}
	\caption{The unique eight-vertex two-sphere triangulation with
		$\sigma_0 = 0$ and property $T_1$. \label{fig:lktight}}
\end{figure}

\medskip
\noindent
{\bf Case $n = 10$}: We need to consider the classification of nine-vertex 
two-sphere triangulations for which we have the following $\sigma_0$-values.

\begin{center}
  \begin{tabular}{|c|c|c|}
	\hline
	$\sigma_0$&$\approx \sigma_0$&\# two-sphere triangulations \\
	\hline
	\hline
$1/21$&$0.0476$&$1$\\
\hline
$2/21$&$0.0952$&$1$\\
\hline
$8/63$&$0.1269$&$1$\\
\hline
$23/126$&$0.1825$&$2$\\
\hline
$3/14$&$0.2142$&$1$\\
\hline
$2/9$&$0.2222$&$1$\\
\hline
$31/126$&$0.246$&$1$\\
\hline
$2/7$&$0.2857$&$7$\\
\hline
$5/14$&$0.3571$&$11$\\
\hline
$1/2$&$0.5$&$24$\\
	\hline
	\hline
	& Total:& $50$ \\
	\hline
  \end{tabular}
\end{center}

In particular the smallest $\sigma_0$-value is greater than zero and thus
Equation~(\ref{eq:bd}) can never hold. Hence no tight ten-vertex 
combinatorial three-manifold $M$ with $\beta_1(M,\mathbb{F})=1$ exist.

\subsubsection*{The case $\beta_1(M,\mathbb{F})=2$}

Again, we have to consider two cases.

\medskip
\noindent
{\bf Case $n=11$}: The classification of $11$-vertex combinatorial three-manifolds
tells us that there is no combinatorial three-manifold with $\leq 11$
vertices and first Betti number greater than one.

\medskip
\noindent
{\bf Case $n=12$}: There are $1249$ triangulations of the sphere with
$11$-vertices. A computer search testing all $1249$ triangulations of the two-sphere with $11$ vertices resulted
in $22$ triangulations with property $T_2$ having $18$ distinct $\sigma_0$ values. These
$18$ $\sigma_0$-values allow $29$ combinations of size $12$ with an average of $\beta_1 (M,\mathbb{F}) -1 = 1$, 
resulting in $50$ combinations of vertex links with this property. Amongst these $50$ combinations
of vertex links, $42$ have at least one vertex degree occurring an odd number of times. Such
a combination cannot be the set of vertex links of a closed combinatorial manifold since
pairs of vertex stars in the vertex links have to meet in an edge link of the surrounding manifold and
hence the number of vertices of a given degree over all vertex links of a triangulation must always be an even number.
This leaves us with eight potential combinations of vertex links consisting of $11$ distinct sphere triangulations
(see table below).

\begin{center}
  \begin{tabular}{|l|c|c|}
	\hline
	Isomorphism signature**&$\sigma_0$&degree sequence* \\
	\hline
	\hline
	\texttt{cdef.e.gbhag.haibibjbkbkbjbkbkhk}&$2254/1155$&$4^2 5^8 6^1$ \\
	\hline
	\texttt{cdef.e.gbhag.haibibjbkbjbjbkfjdk}&$2296/1155$&$4^3 5^6 6^2$ \\
	\hline 
	\texttt{cdef.e.fbgbhaibjbi.hajcjdkbkbkbk}&$2323/1155$&$4^3 5^6 6^2$ \\
	\hline 
	\texttt{cdef.e.fbgbgahbibjbkbj.iakbkeiej}&$2367/1155$&$4^4 5^4 6^3$ \\
	\hline 
	\texttt{cdef.e.fbgbhaibjbi.hajckbkakbkbk}&$2370/1155$&$4^4 5^5 6^1 7^1$ \\
	\hline 
	\texttt{cdef.e.fbgbgahbibhbjciajbkbkbkbk}&$2416/1155$&$4^4 5^4 6^3$ \\
	\hline
	\texttt{cdef.e.fbgbgahbibhbjckaj.kajejbk}&$2416/1155$&$4^4 5^5 6^1 7^1$ \\
	\hline 
	\texttt{cddeafbgaf.haibi.hajbjbjbkbkekbk}&$2422/1155$&$3^1 4^1 5^7 6^2$ \\
	\hline 
	\texttt{cdef.e.fbgbgahbhbibjbibkbjckakek}&$2448/1155$&$4^4 5^5 6^1 7^1$ \\
	\hline 
	\texttt{cddeafbgaf.haibi.hajbjbkbkbkbkbk}&$2454/1155$&$3^1 4^2 5^5 6^3$ \\
	\hline 
	\texttt{cddeafbgaf.gahbhbibibjbkbjbkbkek}&$2564/1155$&$3^1 4^2 5^5 6^3$ \\
	\hline
	\hline
  \end{tabular}
\end{center}

\small
\noindent
* The degree sequence $d_1^{e_1} d_2^{e_2} \ldots d_m^{e_m}$ of a two-sphere triangulation $S$
denotes that $S$ has $e_i$ vertices of degree $d_i$, $1 \leq i \leq m$.

\noindent
** The isomorphism signature of a combinatorial manifold uniquely determines its isomorphism
type, i.e., two combinatorial manifolds have equal isomorphism signature if and only if
they are isomorphic. The isomorphism signature given in this table coincides with the one used
by simpcomp \cite{simpcomp,simpcompISSAC,simpcompISSAC11}. Use the function
\texttt{SCFromIsoSig(...)} to generate the complexes. See the manual for details.

\normalsize
\smallskip
For the remaining eight combinations of twelve $11$-vertex two-sphere triangulations
we apply an exhaustive search for tight combinatorial three-manifolds having any of these 
combinations as their vertex links.
The search essentially fixes an ordering of the vertex links
using the fact that by the two-neighbourliness the intersection of any two pairs of vertex stars
contains an edge star of the three-manifold. Then it starts combining
vertex stars in this ordering looping over all matching pairs of vertices in two consecutive links
(i.e., vertices of equal degree in the two-sphere triangulation),
all rotations and reflections of a matching pair and all permutations of the remaining vertices.
In each step the complex is tested if {\em (i)} it is the sub-complex of a two-neighbourly
three-manifold and {\em (ii)} if it satisfies a generalised version of 
property $T_k$. Whenever a complex is valid an additional link is added to the existing combination. If 
it fails one of the tests the link added last is removed and the next option is tried. The algorithm 
in full detail is available from the author upon request.

The above search yielded zero tight combinatorial $12$-vertex three-manifolds and hence we have
the following.

\begin{theorem}
	There is no tight $12$-vertex combinatorial three-manifold.
\end{theorem}

Corollary~\ref{kor:main} follows instantly from the above observations combined with Theorem~\ref{thm:dimOdd}.

\subsubsection*{The case $\beta_1(M,\mathbb{F}) = 3$}

For $\beta_1(M,\mathbb{F}) = 3$ any tight combinatorial three-manifold has to have either
$13$ or $14$ vertices. For $14$ vertices an analysis if the $\sigma_0$ values
of all $13$-vertex triangulations of the two-sphere results in a 
minimal $\sigma_0$-value of $26971/12870 \sim 2.09565 > \beta_1 (M,\mathbb{F}) -1$ and hence
no tight $14$-vertex combinatorial three-manifold with first Betti number three can exist.


The $13$-vertex case has to be left open at this point.
However, it follows from the above observation together 
with the Lutz-K\"uhnel conjecture \cite[Conjecture 1.3]{Kuehnel99CensusTight} 
and a statement conjectured in \cite[Conjecture 32]{Lutz05TrigMnfFewVertCombMnf}.

\begin{kor}
	\label{cor:three}
	The only tight combinatorial three-manifolds with $ \beta_1(M,\mathbb{F}) \leq 3$
	for any field $\mathbb{F}$ are the boundary of the four-simplex and the 
	nine-vertex triangulation of $S^1 \dtimes S^2$.
\end{kor}

\begin{proof}
By the above, a tight combinatorial three-manifold $M$ with $\beta_1 (M,\mathbb{F}) = 3$
must have $13$ vertices. The most natural candidates for the topological type of such a tight 
combinatorial three-manifold are $(S^2 \times S^1)^{\# 3}$ and $(S^2 \dtimes S^1)^{\# 3}$. However, these
manifolds admit non-neighbourly $13$-vertex triangulations \cite[Corollary 32]{Lutz08FVec3Mnf} and thus 
a tight triangulation of these manifolds cannot be strongly minimal.
The statement now follows as a consequence of the Lutz-K\"uhnel conjecture 
\cite[Conjecture 1.3]{Kuehnel99CensusTight} and 
\cite[Conjecture 35]{Lutz05TrigMnfFewVertCombMnf} where it is conjectured that the only 
manifolds with $\beta_1 (M,\mathbb{F}) = 3$ admitting a $13$-vertex triangulation are
$(S^2 \times S^1)^{\# 3}$ and $(S^2 \dtimes S^1)^{\# 3}$.
\end{proof}

In particular, Corollary~\ref{cor:three} implies that modulo \cite[Conjecture 1.3]{Kuehnel99CensusTight} 
and \cite[Conjecture 32]{Lutz05TrigMnfFewVertCombMnf} there is no tight triangulation of the three-torus.


	{\footnotesize
	 \bibliographystyle{abbrv}
	 \bibliography{/home/jonathan/bibliography/bibliography}
	}

\end{document}